\documentclass[12pt]{amsproc}
\usepackage{amssymb}
\usepackage{amsmath}
\usepackage{amsfonts}
\usepackage{geometry}
\usepackage{graphicx}
\providecommand{\U}[1]{\protect\rule{.1in}{.1in}}
\theoremstyle{plain}

\newtheorem{corollary}{Corollary}

\newtheorem{lemma}{Lemma}

\newtheorem{remark}{Remark}

\newtheorem{theorem}{Theorem}
\numberwithin{equation}{section}

\newcommand{\mau}{\mathcal{M}^+_{\lambda,\Lambda}(D^2u)}
\newcommand{\mav}{\mathcal{M}^+_{\lambda,\Lambda}(D^2v)}

\begin{document}
\title[Liouville theorems and fully nonlinear elliptic equations]{Liouville-type theorems  for fully nonlinear elliptic equations and systems  in half spaces}
\author{Guozhen Lu and Jiuyi Zhu  }
\address{Guozhen Lu \\Department of Mathematics\\
Wayne State University\\
Detroit, MI 48202, USA\\
Emails: gzlu@math.wayne.edu}
\address{ Jiuyi Zhu  \\
Department of Mathematics\\
Wayne State University\\
Detroit, MI 48202, USA\\
Emails:  jiuyi.zhu@wayne.edu }
\thanks{\noindent Research is partly supported by a US NSF grant.\\ \noindent Revised version on July 23, 2012.}
\date{}
 \dedicatory{ }

\begin{abstract}In \cite{LWZ}, we established Liouville-type theorems and decay estimates for solutions of a class of high order elliptic
equations and systems without the boundedness assumptions on the solutions. In this paper, we continue our work in
\cite{LWZ} to investigate the role of boundedness assumption in
proving Liouville-type theorems for fully nonlinear equations. We
remove the boundedness assumption of solutions which was often required in
the proof of Liouville-type theorems for fully nonlinear elliptic
equations or systems in half spaces. We also prove the
Liouville-type theorems for supersolutions of a system of fully
nonlinear equations with Pucci extremal operators in half spaces.

\end{abstract}

\maketitle
\section{Introduction}
The article is devoted to the study of Liouville-type theorems for
nonnegative viscosity solution or supersolutions of a class of fully
nonlinear uniformly elliptic equations and systems in a half space
$\mathbb R^n_+$, i.e. either
\begin{equation}
\left\{
\begin{array}{lll}
F(x, D^2u)+u^p= 0 \quad \quad &\mbox{in} \ \mathbb R^n_+,\\
u=0 \quad \quad &\mbox{on} \ \partial  \mathbb R^n_+
\end{array}
\right. \label{bin}
\end{equation}
or
\begin{equation}
\left\{
\begin{array}{lll}
F(x, D^2u)+v^p =0 \quad \quad &\mbox{in} \ \mathbb R^n_+,\\
F(x, D^2v)+u^q =0\quad \quad &\mbox{in}  \ \mathbb R^n_+
\end{array}
\right. \label{bin1}
\end{equation}
where $\mathbb R^n_+=\{x=(x', x_n)\in \mathbb R^{n-1}\times \mathbb
R| x_n>0\}$ with $n\geq 2$. A continuous function $F: \mathbb
R^n\times S_n\to \mathbb R $ is referred as an uniformly elliptic
equation with ellipticity $0<\lambda\leq \Lambda$ if for all $M,
P\in S_n$ with $P\geq 0$ (nonnegative definite), it holds that
\begin{equation}
\lambda tr(P)\leq F(x, M+P)-F(x,M)\leq \Lambda tr(P), \label{unf}
\end{equation}
where $S_n$ is the space of all real symmetric $n\times n$ matrix, and $tr(P)$ is the trace of
$P\in S_n$.

 Liouville-type theorems are
powerful tools in proving a priori bounds for nonnegative solutions
in a bounded domain. They are widely applied in obtaining a priori
estimate for solutions of elliptic equations in the literature.
Using the "blow-up" method (also called rescaling argument)
\cite{GS}, an equation in a bounded domain will blow up into another
equation in the whole Euclidean space or a half space. With the aid
of the corresponding Liouville-type theorem in the Euclidean space
$\mathbb R^n$ and half space $\mathbb R^n_+$ and a contradiction
argument, the a priori bounds could be readily derived. Moreover,
the existence of nonnegative solutions to elliptic equations is
established by the topological degree method using a priori
estimates (see. e.g. \cite{DLN}) .

 In this paper we mainly consider the model in the case that $F(x,
 D^2u)=\mathcal{M}^+_{\lambda,\Lambda}(D^2u)$. Here $
 \mathcal{M}^+_{\lambda,\Lambda}(D^2u)$ is the Pucci extremal
 operator with parameters $0<\lambda\leq\Lambda$, defined by
 $$\mathcal{M}^+_{\lambda,\Lambda}(M)=\Lambda\sum_{e_i>0}
 e_i+\lambda \sum_{e_i<0}
 e_i$$
for any symmetric $n\times n$ matrix $M$, where $e_i=e_i(M), i=1,
\cdots, n$, denotes the eigenvalue of $M$. While
$\mathcal{M}^-_{\lambda,\Lambda}(M)$ is defined as
$$\mathcal{M}^-_{\lambda,\Lambda}(M)=\lambda\sum_{e_i>0}
 e_i+\Lambda \sum_{e_i>0}
 e_i.$$
Pucci's operators are extremal in the sense that
$$\mathcal{M}^+_{\lambda,\Lambda}(M)=\sup_{A\in\mathcal{A}_{\lambda,
\Lambda}}tr(AM),$$
$$\mathcal{M}^-_{\lambda,\Lambda}(M)=\inf_{A\in\mathcal{A}_{\lambda,
\Lambda}}tr(AM)$$ with
$$\mathcal{A}_{\lambda, \Lambda}=\{ A \in S_n: \lambda|\xi|^2\leq
A\xi\cdot\xi^T\leq \Lambda|\xi|^2, \ \forall \xi\in \mathbb R^n
\}.$$
 If
the operator $F$ is uniformly elliptic with ellipticity constant
$0<\lambda\leq \Lambda$, it results in
$$\mathcal{M}^-_{\lambda,\Lambda}(M)\leq F(x, M)\leq
\mathcal{M}^+_{\lambda,\Lambda}(M)$$ when $F(x,O)=0$. We refer to
the monograph \cite{CC} for more details on these operators. Notice
that $\mathcal{M}^+_{\lambda,\Lambda}$ and
$\mathcal{M}^-_{\lambda,\Lambda}$ are not in the divergence form.

When $\lambda=\Lambda=1$, $\mathcal{M}^{\pm}_{\lambda,\Lambda}$
coincide with the Laplace operators. Then (\ref{bin}) with $F(x,
D^2u)$ replaced by  $\mathcal{M}^{\pm}_{\lambda,\Lambda}(D^2u)$
becomes the
\begin{equation}
\triangle u+u^p=0 \quad \quad \mbox{in} \ \mathbb R^n_+. \label{tri}
\end{equation}
It is well known that (\ref{tri}) does not have positive
supersolutions in the half space for $1<p<\frac{n+1}{n-1}$, and does
not have nonnegative solution for $1<p<\frac{n+2}{n-2}$ with $u$
vanishing on the boundary.

In view of these results for the semilinear equation (\ref{tri}), it
would be interesting to understand the structure of solutions for
(\ref{bin}) and (\ref{bin1}). Unlike in the case of the semilinear
equations, the popular technique of Kelvin transform with moving
plane method is no longer available. We also note that there is no
variational structure for fully nonlinear elliptic equations, even
for the Pucci extremal operators. Those impose new difficulties for
studying Liouville-type results. In \cite{CL}, Cutri and Leoni
establish the following non-existence results in the spirit of the
Hadamard three circle theorem \cite{PW}. In particular, they have
also shown that the critical exponent
$$p^+:=\frac{\tilde{n}}{\tilde{n}-2}$$ is optimal for supersolutions
in (\ref{lio}), where
$$\tilde{n}=\frac{\lambda}{\Lambda}(n-1)+1.$$ It exhibits a nontrivial solution
for (\ref{lio}) if $p>p^+$. Namely, it is stated as the following
lemma.

\begin{lemma}
Assume that $n\geq 3$. If $1<p\leq p^+$ or ($1<p<\infty$ if
$\tilde{n}\leq 2$), then the only viscosity supersolution of
\begin{equation}
\left \{\begin{array}{ll} \mau+u^p =0 \quad \ &\mbox{in} \ \
\mathbb{R}^n,\\
u\geq 0 \quad \ &\mbox{in} \ \ \mathbb{R}^n \\
\end{array} \right.
\label{lio}
\end{equation}
is $u\equiv0$. \label{eur}
\end{lemma}

With the help of moving plane method and the above Liouville-type
theorem, Quaas and Sirakov \cite{QS} make use of the idea of
\cite{D} and obtain a Liouville-type result in a half space. They
first prove the solution of (\ref{hal}) is non-decreasing in $x_n$
direction, then it leads to the same problem in $\mathbb R^{n-1}$
after a limiting process, which allows them to use Lemma \ref{eur}. Under
the boundedness assumption, they show that

\begin{lemma}
Let $n\geq 3$ and $\tilde
{p}^+:=\frac{\lambda(n-2)+\Lambda}{\lambda(n-2)-\Lambda}$. Then the
equation
\begin{equation}
\left \{\begin{array}{lll} \mau+u^p =0 \quad \ &\mbox{in} \ \
\mathbb{R}^n_+, \\
u\geq 0 \quad \ &\mbox{in} \ \ \mathbb{R}^n_+,  \\
u=0 \quad \ &\mbox{in} \ \ \partial\mathbb{R}^n_+ \label{hal}
\end{array} \right.
\end{equation}
has no nontrivial bounded solution, provided $1< p\leq \tilde{p}^+$
and $\lambda(n-2)>\Lambda$ ( or $1<p<\infty$ if $\lambda(n-2)\leq
\Lambda$). \label{tha}
\end{lemma}

Note that $\tilde {p}^+> p^+$ for $\lambda(n-2)>\Lambda$.   We are
interested in the boundedness assumption in Lemma \ref{tha}. As we
know, boundedness assumptions are often imposed in deriving such
Liouville-type theorem in half spaces. Using the Doubling Lemma
recently developed in \cite{PQS} (see  Section 2) and a blow-up
technique, we indeed show that the boundedness assumption is
unnecessary for such equations. Similar ideas have been  applied to
derive Liouville type theorems for solutions to  higher order
elliptic equations and systems in our recent paper \cite{LWZ}. Our
strategy is based on a contradiction argument. We suppose that the
solution $u$ in (\ref{hal}) is unbounded. By the Doubling Lemma and
blow-up method, the equation (\ref{hal}) will become an equation in
a whole Euclidean space or a half space. We will then arrive at a
contradiction under a certain range of $p$, which means that the
solution $u$ has to be bounded. Applying Lemma \ref{tha} again, we
obtain the Liouville-type results. In this paper, we first obtain
the following result.

\begin{theorem}
Let $n\geq 3$. For $1<p\leq p^+$ if $\tilde{n}>2$ (or $1<p< \infty$
if $\tilde{n}\leq 2$ ), then the only nonnegative solution for
(\ref{hal}) is $u\equiv 0$. \label{th1}
\end{theorem}

Quaas and Sirakov in \cite{QS1} consider the non-existence results
for  the elliptic system with Pucci extremal operators  in the Euclidean
space and a half space, which are essential in getting a priori bound
and existence by fixed point theorem for fully nonlinear elliptic
system. Motivated by the work \cite{CL}, they characterized the
range of powers $p, q$ for the nonexistence of positive supersolutions of  (\ref{sysg}) in the Euclidean space.

\begin{lemma}
Let $\lambda_1, \lambda_2, \Lambda_1,\Lambda_2>0$. Set
$$\mathcal{M}^+_l(D^2u_l)=\mathcal{M}^+_{\lambda_l,\Lambda_l}(D^2u_l)$$
for $l=1,2$. Define $$\rho_l=\frac{\lambda_l}{\Lambda_l}, \quad
N_l=\rho_l(n-1)+1.$$  Let $N_1, N_2>2$ and $pq>1$ with $p, q\geq 1$.
Then there are no positive supersolutions for
\begin{equation}
\left \{ \begin{array}{ll}
\mathcal{M}^+_{1}(D^2u_1)+u_2^p=0   \quad \quad  &\mbox{in}\ \mathbb R^n, \\
\mathcal{M}^+_{2}(D^2u_2)+u_1^q=0   \quad \quad  &\mbox{in}\ \mathbb
R^n, \\
\end{array}
\right. \label{sysg}
 \end{equation}
if
$$\frac{2(p+1)}{pq-1}\geq N_1-2, \ \mbox{or} \ \
\frac{2(q+1)}{pq-1}\geq N_2-2.$$ \label{pos}
\end{lemma}

By the moving plane method and Lemma \ref{pos} in the Euclidean space, the
following Liouville-type theorem in a half space is also established
under the boundedness assumption in \cite{QS1}.
\begin{lemma}
Let $N_1, N_2>2$ and $pq>1$ with $p, q\geq 1$. There exist no
positive bounded solutions for the elliptic equation system
\begin{equation}
\left \{ \begin{array}{ll}
\mathcal{M}^+_{1}(D^2u_1)+u_2^p=0   \quad \quad  &\mbox{in}\ \mathbb R^n_+, \\
\mathcal{M}^+_{2}(D^2u_2)+u_1^q=0   \quad \quad  &\mbox{in}\ \mathbb
R^n_+, \\
u_1=u_2=0    \quad  &\mbox{on} \ \partial \mathbb
 R^n_+,
\end{array}
\right. \label{sysf}
\end{equation}
provided
\begin{equation}
\frac{2(p+1)}{pq-1}\geq N_1-2, \ \mbox{or} \ \
\frac{2(q+1)}{pq-1}\geq N_2-2. \label{exp}
\end{equation}
\label{lem2}
\end{lemma}

We are also able to get rid of the boundedness assumption in the above lemma by
choosing appropriate rescaling functions and employing the Doubling
Lemma argument. More precisely,  we prove the following

\begin{theorem}
There exist no positive solutions for (\ref{sysf}) if $p, q>1$ and
the assumption (\ref{exp}) is satisfied. \label{th2}
\end{theorem}

With the Liouville-type theorem for the Euclidean space in hand and the
Doubling Lemma, we can further investigate the singularity and decay
estimates for positive solutions of fully nonlinear elliptic
equations in a  bounded domain or an  exterior domain. Let $1<p\leq
p^+$ if $\tilde{n}>2$ or $1<p<\infty$ if $\tilde{n}\leq 2$. Recall
that $\tilde{n}=\frac{\lambda}{\Lambda}(n-1)+1$. We consider
\begin{equation}
\mathcal{M}^+_{\lambda,\Lambda}(D^2u)+u^p=0   \quad \quad \mbox{in}\
\Omega.
 \label{ome}
\end{equation}
We will establish the following
\begin{theorem}
Let $\Omega\not=\mathbb R^n$ be a domain in $\mathbb R^n$. There
exists $C=C(n,p)>0$ such that any nonnegative solution of (\ref{ome})
satisfies
\begin{equation} u+|\nabla u|^{\frac{2}{p+1}}\leq C
dist^{\frac{-2}{p-1}}(x,\,\partial\Omega), \quad \ \forall \
x\in\Omega.\label{sin}
\end{equation}

In particular, if $\Omega$ is an exterior domain, i.e. the set
$\{x\in\mathbb R^n| |x|>R \}$ for some $R>0$, then
$$    u+|\nabla u|^{\frac{2}{p+1}}\leq C
|x|^{\frac{-2}{p-1}}, \quad \ \forall \ |x|\geq 2R.
$$
\label{thh}
\end{theorem}

If there exists a solution for a general continuous function $f(u)$, i.e. $u$
is a nonnegative solution for
\begin{equation} \mathcal{M}^+_{\lambda, \Lambda}(D^2
u)+f(u)= 0 \quad \ \mbox{in} \ \Omega. \label{gera}
\end{equation}
Similar singular and decay estimates also hold. Namely, if $1<p\leq
p^+$ for $\tilde{n}>2$ or $1<p<\infty$ for $\tilde{n}\leq 2$, we
have the following corollary.

\begin{corollary}
Assume that  $$\lim_{u\to \infty} u^{-p}f(u)=\gamma\in (0,\infty).$$
There exists $C(n,f)>0$ independent of $\Omega$  such that any
positive solution  in (\ref{gera}) satisfies
$$ u+|\nabla
u|^{\frac{2}{p+1}}\leq C
(1+dist^{\frac{-2}{p-1}}(x,\,\partial\Omega)), \quad \ \forall \
x\in\Omega.$$

In particular, if $\Omega=\mathbb B_R\backslash \{0\}$ for some $R$,
then
$$    u+|\nabla u|^{\frac{2}{p+1}}\leq C
(1+|x|^{\frac{-2}{p-1}}), \quad \ \forall \ 0<|x|\leq R/2.
$$
\label{cor1}
\end{corollary}

\begin{remark}
The similar results also hold for $\mathcal{M}^-_{\lambda,
\Lambda}(D^2u)$ and its system in Theorem \ref{th1}, Theorem
\ref{th2} and Theorem \ref{thh}.
\end{remark}

The study of the supersolutions for
\begin{equation}
\mathcal{M}_{\lambda,\Lambda}^-(D^2u)+u^p= 0 \quad \ \mbox{in} \
\mathbb R^n_+ \label{leo}
\end{equation}
without assumed boundary condition is  more involved. Recently, Leoni
\cite{L} obtained the Liouville-type results for (\ref{leo}),  that
is, there does not exist any positive solution in (\ref{leo}) for
$-1\leq p\leq \frac{\Lambda n+\lambda}{\Lambda n-\lambda}$. By
explicit test functions, there does exist a supersolution for
$p>\frac{\Lambda(n-1)+2\lambda}{\Lambda(n-1)}$, which is considered to be
the critical exponent for Liouville-type property \cite{AS}. The
existence or non-existence of any solution for (\ref{leo}) is still
unknown for $$\frac{\Lambda n+\lambda}{\Lambda n-\lambda}<p\leq
\frac{\Lambda(n-1)+2\lambda}{\Lambda(n-1)}.$$ In \cite{L}, the
author also points out that the inequality
\begin{equation}
\mau+u^p\leq 0  \quad \quad \mbox{in} \ \mathbb R^n_+
\end{equation}
does not have any positive solution for $$-1\leq p\leq
\frac{\tilde{n}+1}{\tilde{n}-1}.$$  Adapting the idea in \cite{L},
we consider the supersolutions for a system of fully nonlinear
elliptic equations with Pucci's extremal operators in half spaces,
i.e.
\begin{equation}
\left \{ \begin{array}{ll}
\mau+v^p=0   \quad \quad  &\mbox{in}\ \mathbb R^n_+, \\
\mav+u^q=0   \quad \quad  &\mbox{in}\ \mathbb
R^n_+. \\
\end{array}
\right. \label{pus}
\end{equation}
The difficulty of Leoni' proof in \cite{L} for (\ref{leo}) is to
show the Liouville-type property holds for the limiting case $p=
\frac{\Lambda n+\lambda}{\Lambda n-\lambda}$. In order to achieve
this, some explicit subsolution is constructed under complicated
calculations. Our main effort is also devoted to building such
explicit subsolution for the operator $\mathcal{M}_{\lambda,
\Lambda}^+$ instead of $\mathcal{M}_{\lambda, \Lambda}^-$. We show
the following Liouville-type theorem:
\begin{theorem}
Assume that $\tilde{n}\geq 2$ and $p,q>0$, there does not exist any
nontrivial nonnegative supersolution in (\ref{pus}) provided

(1) $pq>1$ and $\frac{2(p+1)}{pq-1}> \tilde{n}-1$ or
$\frac{2(q+1)}{pq-1}> \tilde{n}-1,$ \\
or

(2) $\frac{2(p+1)}{pq-1} =\tilde{n}-1$ and $\frac{2(q+1)}{pq-1}=
\tilde{n}-1,$ \\
or

(3) $pq=1$. \label{th4}

\end{theorem}

Combining our idea in Theorem \ref{th4} and the estimates for
$\mathcal{M}_{\lambda, \Lambda}^-(D^2 u)$ in \cite{L}, we are able
to establish the following Liouville-type results for
\begin{equation}
\left \{ \begin{array}{ll}
\mathcal{M}^-_{\lambda, \Lambda}(D^2u)+v^p=0   \quad \quad  &\mbox{in}\ \mathbb R^n_+, \\
\mathcal{M}^-_{\lambda, \Lambda}(D^2v)+u^q=0   \quad \quad
&\mbox{in}\ \mathbb
R^n_+. \\
\end{array}
\right. \label{puc}
\end{equation}

\begin{corollary}
There exists only trivial nonnegative supersolution for (\ref{puc})
if

(1) $pq>1$ and $\frac{2(p+1)}{pq-1}> \frac{\Lambda n}{\lambda}-1$ or
$\frac{2(q+1)}{pq-1}>  \frac{\Lambda n}{\lambda}-1,$\\
 or

(2)$\frac{2(p+1)}{pq-1}=\frac{\Lambda n}{\lambda}-1$ and
$\frac{2(q+1)}{pq-1}=\frac{\Lambda n}{\lambda}-1,$\\
or

(3) $pq=1$. \label{cor2}

\end{corollary}

Finally we note that there is a large literature concerning
Liouville-type results for solution (or supersolution, or
subsolution) of elliptic equations or system. We refer to
\cite{AS1}, \cite{CC1}, \cite{CL1}  \cite{DM}, \cite{FQ},
\cite{GS1}, \cite{LZ}, \cite{SZ} and references therein for more
account.

 The outline of the paper is as follows. In
Section 2, we present the basic results for the definition of
viscosity solution, comparison principle, Doubling Lemma and so on.
Section 3 is devoted to the proof of removing the boundedness
assumption for fully nonlinear elliptic equations and systems. We
also show the singularity and decay estimates for a single equation.
The Liouville-type theorem for a system of equations in a half space
without boundary assumption is considered in Section 4.  Throughout
the paper, $C $ and $C_1$  denote generic positive constants, which
are independent of $u$, $v$ and may vary from line to line.

\section{Preliminaries}
 In this section we collect some basic results which will be applied
 throughout the paper for fully nonlinear elliptic equations. We refer to
 \cite{CC}, \cite{CL}, \cite{QS} and  references therein for the proofs and results.

Let us recall the notion of viscosity sub and supersolutions of
fully nonlinear elliptic equations
\begin{equation}
F(x, u, D^2u)=0 \quad \ \mbox{in} \ \Omega, \label{def}
\end{equation}
where $\Omega$ is an open domain in $\mathbb R^n$ and
$F:\Omega\times \mathbb R\times S_n\to \mathbb R$ is a continuous
map with $F(x, t, M)$ satisfying (\ref{unf}) for every fixed $t\in
\mathbb R$, $x\in\Omega$.

\emph{Definition:} A continuous function $u: \Omega \to \mathbb R$
is a viscosity supersolution (subsolution) of (\ref{def}) in $\Omega$,
when the following condition holds: If $x_0\in\Omega$, $\phi\in
C^2(\Omega)$ and $u-\phi$ has a local minimum (maximum) at $x_0$,
then
$$ F(x_0, \phi(x_0), D^2\phi(x_0))\leq (\geq) 0. $$

If $u$ is a viscosity supersolution (subsolution), we say that $u$
verifies
$$F(x, u, D^2u)\leq(\geq) 0$$
in the viscosity sense.

 We say that $u$ is a viscosity solution of (\ref{def}) when it
 simultaneously
is a viscosity subsolution and supersolution.

We will make use of the following comparison principle (see e.g.
\cite{CL}).
\begin{lemma}
(Comparison Principle) Let $\Omega\in\mathbb R^n$ be a bounded
domain and $f\in C(\Omega)$. If $u$ and $v$ are respectively a
supersolution and subsolution either of
$\mathcal{M}^+_{\lambda,\Lambda}(D^2u)=f(x)$ or of
$\mathcal{M}^-_{\lambda,\Lambda}(D^2u)=f(x)$ in $\Omega$,  and
$u\geq v$ on $\partial\Omega$, then $u\geq v$ in $\bar\Omega$.
\label{com}
\end{lemma}

The following version of the Hopf boundary lemma holds (see e.g.
\cite{QS}).
\begin{lemma}
Let $\Omega$ be a regular domain and $u\in W^{2,n}_{loc}(\Omega)\cap
C(\bar\Omega)$ be a nonnegative solution to
$$\mathcal{M}^+_{\lambda,\Lambda}(D^2u)+c(x)u\leq 0 \quad \ \mbox{in} \ \Omega$$
with bounded $c(x)$. Then either $u\equiv 0$ in $\Omega$ or $u(x)>0$
for all $x\in\Omega$. Moreover, in the latter case for any
$x\in\partial\Omega$ such that $u(x_0)=0$,
$$\lim_{t\to 0^+}\sup\frac{u(x_0-t\nu)-u(x_0)}{t}<0,$$
where $\nu$ is the outer normal to $\partial\Omega$. \label{hopf}
\end{lemma}

We are going to  use the following regularity results in \cite{CC}
for Pucci operators in the blow-up argument.
\begin{lemma}
(Regularity Lemma) If $u$ is a viscosity solution to the fully
nonlinear elliptic equation with Pucci extremal operator
\begin{equation}
\mathcal{M}^+_{\lambda,\Lambda}(D^2u)+g(x)=0 \label{max}
\end{equation}
in a ball $\mathbb B_{2R}$ and $g\in L^{p}(\mathbb B_R)$ for some
$p\geq n$, then $u\in W^{2, p}(\mathbb B_R)$ and the following
interior estimate holds
\begin{equation}
\|u\|_{W^{2,p}(\mathbb B_R)}\leq C (\|u\|_{L^{\infty}(\mathbb
B_{2R})}+\|g\|_{L^p(\mathbb B_{2R})}).
\end{equation}
Furthermore, if $g\in C^{\alpha}$ for some $\alpha\in (0,1)$, then
$u\in C^{2,\alpha}$ and
\begin{equation}
\|u\|_{C^{2,\alpha}(\mathbb B_R)}\leq C (\|u\|_{L^{\infty}(\mathbb
B_{2R})}+\|g\|_{C^{\alpha}(\mathbb B_{2R})}).
\end{equation}
In addition, if (\ref{max}) holds in a regular domain and $u=0$ on
the boundary, then $u$ satisfies a $C^{\alpha}$- estimate up to the
boundary. \label{reg}
\end{lemma}
Note that the above $C^{2,\alpha}$ estimate depends on the convexity
of the Pucci extremal operator. Next we state the closeness of a
family of viscosity solutions to fully nonlinear equations (see e.g.
\cite{CC}).
\begin{lemma}
Assume $u_n$ and $g_n$ are sequences of continuous functions and $u_n$
is a solution (or subsolution, or supersolution) of the equation
$$\mathcal{M}^+_{\lambda,\Lambda}(D^2u_n)+g_n(x)=0 \quad \ \
\mbox{in} \ \ \Omega.$$ Assume that $u_n$ and $g_n$ converge
uniformly on compact subsets of $\Omega$ to function $u$ and $g$.
Then $u$ is a solution (or subsolution, or supersolution) of the
equation
$$\mathcal{M}^+_{\lambda,\Lambda}(D^2u)+g(x)=0 \quad \ \
\mbox{in} \ \ \Omega.$$ \label{cov}
\end{lemma}

We  state the following technical lemma that is frequently used  in
Section 3.  The proof of this lemma is given in \cite{PQS}.
An interested reader may refer to it for more details.  Based on the
doubling property, we can start the rescaling process to prove local
estimates of solutions for fully nonlinear equations.
\begin{lemma}
(Doubling lemma) Let $(X, \,d )$ be a complete metric space and
$\emptyset\not=D\subset\Sigma\subset X,$ with $\Sigma$ closed.
Define $M: D\to (0, \infty)$ to be bounded on compact subsets of
$D$.  If $y\in D$ is such that
$$ M(y)dist(y, \Gamma)>2k$$
for  a  fixed positive number $k$, where $\Gamma =\Sigma \setminus
D,$ then there exists $x\in D$ such that
$$M(x) dist(x, \, \Gamma)> 2k, \ \ \ \ M(x)\geq M(y).$$
Moreover,
$$ M(z)\leq 2M(x), \ \ \ \ \forall z\in D\cap \bar B(x,
kM^{-1}(x)). $$ \label{dob}
\end{lemma}
\begin{remark}
 If $\Gamma=\emptyset,$ then $ dist(x, \Gamma):=\infty$. In this
 case, we have following the version of the Doubling Lemma. Let $D=\Sigma\subset X,$ with $\Sigma$ closed.
Define $M: D\to (0, \infty)$ to be bounded on compact subsets of
$D$, For every $y\in D$, there exists $x\in D$ such that
$$M(x)\geq
M(y)$$ and
$$ M(z)\leq 2M(x), \ \ \ \ \forall z\in D\cap \bar B(x,
kM^{-1}(x)). $$ \label{rem2}
\end{remark}

\section{Liouville-type theorems for elliptic equations in half spaces}
We first present the proof of Theorem \ref{th1}. Our idea is the
combination of doubling property and blow-up argument. This idea seems
to be powerful in getting rid of the boundedness  assumption whenever
proving Liouville-type theorems. We refer to \cite{LWZ} for
applications of this idea in higher order elliptic equations.

\begin{proof}[proof of Theorem \ref{th1}:]
 Suppose that a solution $u$  to the equation (\ref{hal}) is
unbounded.  Namely,  there exists
 a sequence of $(y_k)\in \mathbb R^n_+$ such that $$
 u(y_k)\to \infty$$ as $k\to \infty$. Set
 $$M(y):= u^{\frac{p-1}{2}}(y): \mathbb R_+^n \to \mathbb
 R.$$
 Then $M(y_k)\to \infty$ as $k\to\infty$ by the fact that $p>1$.
 By taking $D=\Sigma=X=\overline{\mathbb R_+^n }$ in
 the Doubling Lemma (i.e. Lemma \ref{dob})
 and Remark \ref{rem2}, there
 exists another sequence of $\{x_k\}$ such that
 $$ M(x_k)\geq M(y_k)$$
 and
 $$ M(z)\leq 2 M(x_k), \ \ \ \forall z\in
 B_{k/M(x_k)}(x_k)\cap \overline{\mathbb R^n_+}.$$
 Set $$d_k:=x_{k, n}M(x_k),$$ where $x_k=(x_{k,1}, \cdots, x_{k, n})$ and
 $$ H_k:=\{\xi=(\xi_1, \cdots, \xi_n)\in\mathbb R^n |\xi_n>-d_k\}.$$
We define  a new function
  $$v_k(\xi):=\frac{u(x_k+\frac{\xi}{M(x_k)})}{M^{\frac{2}{p-1}}(x_k)}.$$
Then, $v_k(\xi) $ is the nonnegative solution of
\begin{equation}
\left \{ \begin{array}{ll}
 \mathcal{M}^+_{\lambda,\Lambda}(D^2v_k)+v_k^p =0   \quad \quad  &\mbox{in}\ H_k, \\
 v_k=0    \quad  &\mbox{on} \ \partial H_k=\{\xi\in\mathbb
 R^n|\xi_n=-d_k\}
\end{array}
\right. \label{ma1}
\end{equation}
with
\begin{equation}
v_k^{\frac{p-1}{2}}(0)=1 \label{eq}
\end{equation} and
\begin{equation}
v_k^{\frac{p-1}{2}}(\xi)\leq 2 ,\ \quad \quad \forall \xi \in
H_k\cap B_k(0). \label{bo}
\end{equation}

Two cases may occur as $k\to \infty$, either Case (1)
$$x_{k, n}M(x_k)\to \infty$$
for a subsequence still denoted as before, or Case  (2)
$$ x_{k, n}M(x_k)\to d$$
for a subsequence still denoted as before, here $d\geq 0$. If Case
(1) occurs, i.e. $ H_k\cap B_k(0)\to \mathbb R^n$ as $ k\to \infty$,
then for any smooth compact set $D$ in $\mathbb R^n$, there exists
$k_0$ large enough such that $D\subset (H_k\cap B_k(0))$ as $k\geq
k_0$. By regularity lemma (i.e. Lemma \ref{reg}),  (\ref{bo}) and
Arzel$\acute{a}$-Ascoli theorem, $v_k\to v$ in $C^2(\bar D)$ for a
subsequence. Furthermore, using a diagonalization  argument, $v_k\to
v$ in $C^{2}_{loc}(\mathbb R^n)$ as $k\to\infty$. From Lemma
\ref{cov}, we know that $v$ solves
$$
 \mav+v^p=0  \quad \ \mbox{in} \ \mathbb R^n.$$ Thanks to Lemma \ref{eur},
 there exists only a trivial solution provided
 \begin{equation}
 1<p\leq p^+  \quad \ \mbox{for} \ \lambda(n-1)>\Lambda
\label{err1}
 \end{equation}
 or
 \begin{equation}
1<p<\infty \quad \ \mbox{for} \ \lambda(n-1)\leq\Lambda
\label{err2}.
 \end{equation}
 In the above,  we have used the fact that $\tilde{n}=2$ is equivalent to  $\lambda(n-1)=\Lambda$.
 However, (\ref{eq}) implies that
\begin{equation}
 v^{\frac{p-1}{2}}(0)=1, \nonumber
\end{equation} which indicates that $v$ is nontrivial.  This
contradiction leads to the conclusion that $u$ in (\ref{hal}) is
bounded in the above range of $p$.

If the Case (2) occurs, we make a further translation. Set

$$ \tilde {v}_k(\xi):=v_k(\xi-d_ke_n) \quad \mbox{for} \
\xi\in\overline{\mathbb R_+^n }.$$ Then  $\tilde {v}_k$ satisfies
\begin{equation}
\left \{ \begin{array}{ll}
\mathcal{M}^+_{\lambda,\Lambda}(D^2\tilde{v}_k)+\tilde{v}_k^p=0   \quad \quad  &\mbox{in}\ \mathbb R^n_+, \\
\tilde{v}_k\geq 0  \quad \quad  &\mbox{in}\ \mathbb R^n_+,\\
 \tilde{v}_k=0    \quad  &\mbox{on} \ \partial \mathbb
 R^n_+.
\end{array}
\right. \label{ma2}
\end{equation}
While
\begin{equation}
 \tilde
{v}_k^{\frac{p-1}{2}}(d_ke_n)=1 \label{eq1}
\end{equation}
and
\begin{equation}
\tilde {v}_k^{\frac{p-1}{2}}(\xi)\leq 2 ,\ \quad \quad \forall \xi
\in \mathbb R^n_+\cap B_k(d_ke_n). \label{bo1}
\end{equation}
For any smooth compact $D$ in $\overline{\mathbb R_+^n }$, there
also exists $k_0$ large enough such that $D\subset
(\overline{\mathbb R_+^n }\cap B_k(0))$ for any $k\geq k_0$. Thanks
to regularity Lemma \ref{reg} and (\ref{bo1}), we can extract a
subsequence of $\tilde {v}_k$ such that  $\tilde {v}_k\to v$ in
$C^2(\bar D)\cap C(\bar D)$. A diagonalization argument shows that
$\tilde {v}_k\to v$ uniformly as $k\to\infty$. Furthermore, by Lemma
\ref{cov}, $v$ solves
\begin{equation}
\left \{ \begin{array}{ll}
\mathcal{M}^+_{\lambda,\Lambda}(D^2v)+v^p=0   \quad \quad  &\mbox{in}\ \mathbb R^n_+, \\
v\geq 0  \quad \quad  &\mbox{in}\ \mathbb R^n_+,\\
v=0    \quad  &\mbox{on} \ \partial \mathbb
 R^n_+.
\end{array}
\right.
\end{equation}
Due to Lemma \ref{tha}, we readily have that $v\equiv 0$ if
\begin{equation}
1<p\leq \tilde{p}^+ \quad \ \mbox{for} \ \lambda(n-2)>\Lambda
\label{err3}
\end{equation}
or
\begin{equation}
1<p<\infty \quad \ \mbox{for} \ \lambda(n-2)\leq\Lambda
\label{err4}.
 \end{equation}
 It contradicts again with the fact that
\begin{equation}
v (de_n)^{\frac{p-1}{2}}=1
\end{equation}
from (\ref{eq1}). Hence $u$ is bounded in  Case (2).

Together with (\ref{err1}), (\ref{err2}), (\ref{err3}) and
(\ref{err4}), we infer that $u$ is bounded in (\ref{hal}) if
$1<p\leq p^+$ in the case of $\lambda(n-1)>\Lambda$ or if
$1<p<\infty$ in the case of $\lambda(n-1)\leq\Lambda$. Note again
that $\tilde{n}=2$ implies that $\lambda(n-1)=\Lambda$. Applying
Lemma \ref{tha} again, we obtain Theorem \ref{th1} in the above
range of $p$.
\end{proof}

We are now in the position to prove Theorem \ref{th2}. Since we consider
the elliptic system with different powers $p, q$, we shall
choose the rescaling function appropriately.

\begin{proof}[Proof of Theorem \ref{th2}:]
Assume by contradiction that either $u_1$ or $u_2$ is unbounded,
that is, there exists a sequence $y_k$ such that
$$M_k(y_k)=u_1^{1/\alpha}(y_k)+u_2^{1/\beta}(y_k)\to \infty$$
as $k\to \infty$. The constant $\alpha, \beta$ are positive numbers
which will be determined later. From the Doubling Lemma and Remark
1, there exists a sequence of  $\{x_k\}$ such that
 $$ M(x_k)\geq M(y_k)$$
 and
 $$ M(z)\leq 2 M(x_k), \ \ \ \forall z\in
 B_{k/M(x_k)}(x_k)\cap \overline{\mathbb R^n_+}.$$
 Define $$d_k:=x_{k, n}M(x_k)$$ and
 $$ H_k:=\{\xi\in\mathbb R^n |\xi_n>-d_k\}.$$
We do the following rescaling,
  $$v_{1,k}(\xi):=\frac{u_1(x_k+\frac{\xi}{M(x_k)})}{M^{\alpha}(x_k)},$$
$$v_{2,k}(\xi):=\frac{u_2(x_k+\frac{\xi}{M(x_k)})}{M^{\beta}(x_k)}.$$

Then, by (\ref{sysf}), $v_{1,k}(\xi)$, $v_{2,k}(\xi)$ satisfy
\begin{equation}
\left \{\begin{array}{ll}
\mathcal{M}^+_1(D^2v_{1,k})M_k^{\alpha+2}(x_k)+M_k^{p\beta}(x_k)v_{2,k}^p=0
\quad  &\mbox{in}\  H_k,\\
\mathcal{M}^+_2(D^2v_{2,k})M_k^{\beta+2}(x_k)+M_k^{q\alpha}(x_k)v_{1,k}^q=0
\quad  &\mbox{in}\  H_k,\\
v_{1,k}=v_{2,k}=0 \quad  &\mbox{in} \ \partial H_k.
\end{array}
\right. \label{get}
\end{equation}
In order to get rid of $M_k(x_k)$ in (\ref{get}), by setting
$\alpha+2=p\beta$ and $\beta+2=q\alpha$, we conclude  that
$$\alpha=\frac{2(p+1)}{pq-1},$$
$$\beta=\frac{2(q+1)}{pq-1}.$$
With so chosen $\alpha, \beta$, then $v_{1,k}, v_{2,k}$ solve
\begin{equation}
\left \{\begin{array}{ll} \mathcal{M}^+_1(D^2v_{1,k})+v_{2,k}^p=0
\quad  &\mbox{in}\  H_k,\\
\mathcal{M}^+_2(D^2v_{2,k})+v_{1,k}^q=0
\quad  &\mbox{in}\  H_k.\\
v_{1,k}=v_{2,k}=0 \quad  &\mbox{in} \ \partial H_k.
\end{array}
\right.
\end{equation}

Furthermore,
\begin{equation}
v_{1,k}^{\frac{1}{\alpha}}(0)+v_{2,k}^{\frac{1}{\beta}}(0)=1
\label{con}
\end{equation}
and
$$v_{1,k}^{\frac{1}{\alpha}}(\xi)+v_{2,k}^{\frac{1}{\beta}}(\xi)\leq 2,
\quad \quad \forall \xi\in H_k\cap \mathbb B_k(0).$$

Two cases
may occur as $k\to\infty$, either Case (1),
$$d_k\to\infty$$
for a subsequence still denoted as before, or Case (2)
$$d_k\to d$$
for a subsequence still denoted as before. We  note that $d\geq
0$.

 If Case (1) occurs, i.e. $H_k\cap \mathbb B_k(0)\to\mathbb R^n$,
we argue similarly as in the proof of Theorem \ref{th1}. For any smooth compact set
$D$ in $\mathbb R^n$, by Lemma \ref{reg} and Arzel$\acute{a}$-Ascoli
theorem, we know that $v_{1, k}\to v_1$ and $v_{2, k}\to v_2$ in
$C^2(\bar D)$ for a  subsequence. Using a diagonalization argument,
$v_{1, k}\to v_1$ and $v_{2, k}\to v_2$ in $C^{2}_{loc}(\mathbb
R^n)$ as $k\to\infty$. From Lemma \ref{cov}, we obtain that  $v_1,
v_2$ satisfy
\begin{equation}
\left \{\begin{array}{ll} \mathcal{M}^+_1(D^2v_1)+v_{2}^p=0
\quad  &\mbox{in}\  \mathbb R^n,\\
\mathcal{M}^+_2(D^2v_{2})+v_{1}^q=0
\quad  &\mbox{in}\  \mathbb R^n.\\
\end{array}
\right.
\end{equation}
 As shown in Lemma \ref{pos}, $v_{1}\equiv v_{2}\equiv 0$ provided
 $$\frac{2(p+1)}{pq-1}\geq N_1-2, \ \mbox{or} \ \
\frac{2(q+1)}{pq-1}\geq N_2-2.$$   Nevertheless, (\ref{con})
indicates that
 either $v_1$ or $v_2$ is nontrivial. We arrive at the
 contradiction, which indicates $u_1, u_2$ in (\ref{sysf}) are actually bounded in Case (1).

If Case (2) occurs, we translate the equation to be in the standard
half space. Let
$$ \tilde {v}_{1,k}(\xi):=v_{1,k}(\xi-d_ke_n) \quad \mbox{for} \
\xi\in\overline{\mathbb R_+^n },$$
$$\tilde{v}_{2,k}(\xi):=v_{2,k}(\xi-d_ke_n) \quad \mbox{for} \
\xi\in\overline{\mathbb R_+^n }.$$

Then  $\tilde {v}_{1,k}$, $\tilde {v}_{2,k}$ satisfy
\begin{equation}
\left \{ \begin{array}{ll}
\mathcal{M}^+_1(D^2\tilde{v}_{1,k})+\tilde{v}_{2,k}^p=0   \quad \quad  &\mbox{in}\ \mathbb R^n_+, \\
\mathcal{M}^+_2(D^2\tilde{v}_{2,k})+\tilde{v}_{1,k}^q=0   \quad \quad  &\mbox{in}\ \mathbb R^n_+, \\
 \tilde{v}_{1,k}= \tilde{v}_{1,k}=0    \quad  &\mbox{on} \ \partial \mathbb
 R^n_+.
\end{array}
\right. \label{ma22}
\end{equation}
Moreover,
\begin{equation}
 \tilde
{v}_{1,k}^{\frac{1}{\alpha}}(d_ke_n)+\tilde
{v}_{2,k}^{\frac{1}{\beta}}(d_ke_n)=1 \label{ide}
\end{equation}
and
\begin{equation}
\tilde{v}_{1,k}^{\frac{1}{\alpha}}(\xi)+\tilde
{v}_{2,k}^{\frac{1}{\beta}}(\xi)\leq 2 ,\ \quad \quad \forall \xi
\in \mathbb R^n_+\cap B_k(d_ke_n) \label{est}.
\end{equation}
Similar argument as in the proof of  Theorem \ref{th1} shows that  there exist
$\tilde {v}_{1,k} $ and $\tilde {v}_{2,k} $ such that $$\tilde
{v}_{1,k}\to \tilde {v}_1$$ and $$\tilde {v}_{2,k}\to \tilde {v}_2$$
in $C^2_{loc}(\mathbb R^n_+)\cap C(\overline{\mathbb R_+^n})$ as
$k\to\infty$. $\tilde {v}_1$ and $\tilde {v}_2$ solve
\begin{equation}
\left \{ \begin{array}{ll}
\mathcal{M}^+_1(D^2\tilde{v}_{1})+\tilde{v}_{2}^p=0   \quad \quad  &\mbox{in}\ \mathbb R^n_+, \\
\mathcal{M}^+_2(D^2\tilde{v}_{2})+\tilde{v}_{1}^q=0   \quad \quad  &\mbox{in}\ \mathbb R^n_+, \\
 \tilde{v}_{1}= \tilde{v}_{1}=0    \quad  &\mbox{on} \ \partial \mathbb
 R^n_+.
\end{array}
\right. \label{mm}
\end{equation}

 Lemma \ref{lem2} and (\ref{est}) yield that $\tilde{v}_{1}\equiv \tilde{v}_{2}\equiv 0$ when
(\ref{exp}) holds. However, it contradicts to the fact of (\ref{ide}).

In conclusion, we obtain that $u$ is bounded in (\ref{sysf}) when the exponents $p$ and $q$ satisfy
(\ref{exp}). From  Lemma \ref{lem2} again, we conclude that the
boundedness assumption is not essential, i.e. Theorem \ref{th2}
holds.
\end{proof}

With the help of Lemma \ref{eur} and the Doubling Lemma, we are
ready to give the proof of Theorem \ref{thh}.
\begin{proof}[Proof of Theorem \ref{thh}]
We also argue by contradiction. Suppose that (\ref{sin}) is false.
Then, there exists a sequence of functions $u_k$ in (\ref{ome}) on $\Omega_k$ such that
$$M_k=u_k^{\frac{p-1}{2}}+|\nabla u_k|^{\frac{p-1}{p+1}} $$
satisfying
$$M_k(y_k)>2k dist^{-1}(y_k, \, \partial\Omega_k).$$
By the Doubling Lemma, there exists $x_k\in\Omega_k$ such that
$$M_k(x_k)\geq M_k(y_k),$$
$$M_k(x_k)>2k dist^{-1}(x_k,\, \partial\Omega_k)$$
and
$$M_k(z)\leq 2M_k(x_k),  \quad \mbox{if}  \ |z-x_k|\leq
kM_k^{-1}(x_k).$$ We introduce a rescaled function
$$ v_k(\xi)=\frac{u_k(x_k+\frac{\xi}{M_k(x_k)})}{M_k^{\frac{2}{p-1}}}. $$
Simple calculation yields that
\begin{equation}
\mathcal{M}^+_{\lambda,\Lambda}(D^2v_k)+v^p_k=0, \quad \quad \forall
|\xi|\leq k.
\end{equation}
Moreover,
\begin{equation}
(v_k^{\frac{p-1}{2}}+|\nabla v_k|^{\frac{p-1}{p+1}})(0)=1 \label{bd}
\end{equation}
and \begin{equation}(v_k^{\frac{p-1}{2}}+|\nabla
v_k|^{\frac{p-1}{p+1}})(\xi)\leq 2, \quad \quad \forall |\xi|\leq k.
\label{bdd}
\end{equation}

For any smooth compact set $D$ in $\mathbb R^n$, there exists $k_0$
large enough such that $D\subset \mathbb B_k(0)$ as $k\geq k_0$. By
Lemma \ref{reg} and (\ref{bdd}), we have $$\|v_k\|_{C^{2,
\alpha}(D)}\leq C$$ for some $C>0$. From  Arzel$\acute{a}$-Ascoli
theorem, up to a subsequence, $v_k\to v$ in $C^2(\bar D)$. In
addition, by  a diagonalization argument and Lemma \ref{cov}, $v_k\to
v$ in $C^{2}_{loc}(\mathbb R^n)$ as $k\to \infty$, which solves
$$
 \mav+v^p=0  \quad \ \mbox{in} \ \mathbb R^n.$$ Since $1<p\leq p^+$, Lemma \ref{eur} implies that
 the only solution is $v\equiv 0$. However, (\ref{bd}) shows that
 $v$ is impossible to be trivial. Therefore, this contradiction leads to the conclusion in
 Theorem \ref{thh}.
\end{proof}

For the proof of Corollary \ref{cor1}, it is very similar to the
above argument. We shall omit it here. The interested reader may refer to the
above proof and \cite{PQS}.

\section{A Liouville-type theorem for supersolutions of elliptic systems in a half space}

We introduce the following algebraic result in \cite{L} for the
eigenvalue of a special symmetric matrix.
\begin{lemma}
Let $\nu, \omega\in \mathbb R^n$ be unitary vectors and $a_1, a_2,
a_3$ and $a_4$ be constants. For the symmetric matrix,
$$
A=a_1\nu\otimes\nu+a_2\omega\otimes\omega+a_3(\nu\otimes\omega+\omega\otimes\nu)+a_4I_n,$$
where $\nu\otimes\omega$ denotes the $n\times n$ matrix whose $i, j$
entry is $\nu_i\omega_j$,  the eigenvalues of $A$ are given as
follows,

\indent $\bullet a_4,$ with multiplicity (at least) $n-2$.\\
 \\
\indent$\bullet a_4+\frac{a_1+a_2+2a_3\nu\cdot\omega\pm
\sqrt{(a_1+a_2+2a_3\nu\cdot\omega)^2+4(1-(\nu\cdot\omega))^2(a_3^2-a_1a_2)^2}
}{2}$, which are simple (if different from $a_4$).

In particular, if either $a_3^2=a_1a_2$ or $(\nu\cdot\omega)^2=1$,
then the eigenvalues are $a_4$ with multiplicity $n-1$ and
$a_4+a_1+a_2+2a_3\nu\cdot\omega$, which is simple. \label{tec}
\end{lemma}

Let us consider a lower semicontinuous function $u\in
\overline{\mathbb R^n_+}\to [0, \ \infty)$ for
\begin{equation}
\mathcal{M}^+_{\lambda, \Lambda} (D^2 u)\leq 0 \quad \ \mbox{in} \
\mathbb R^n_+ \label{sup}
\end{equation}
in viscosity sense. For any $r>0$, we define the function
\begin{equation}
m_u(r)=\inf_{\mathbb B^+_r}\frac{u(x)}{x_n}, \label{mon}
\end{equation}
where $\mathbb B^+_r$ is the half ball centered at the origin with
radius $r$ in $\mathbb R^n_+$. We present the following three --
circles Hadamard type results for superharmonic functions in
\cite{L}.
\begin{lemma}
Let $u\in \overline{\mathbb R^n_+}\to [0, \ \infty)$ be  a lower
semicontinuous function satisfying (\ref{sup}). Then the function
$m_u(r)$ in (\ref{mon}) is a concave function of $r^{-\tilde{n}}$,
i.e. for every fixed $R>r>0$ and for all $r\leq\rho\leq R$, one has
\begin{equation}
m_u(\rho)\geq\frac{
m_u(r)(\rho^{-\tilde{n}}-R^{-\tilde{n}})+m_u(R)(r^{-\tilde{n}}-\rho^{-\tilde{n}})}{r^{-\tilde{n}}-R^{-\tilde{n}}}
\label{key}.
\end{equation}
Consequently,
$$ r\in (0, \ \infty)\to m_u(r)r^{\tilde{n}} $$
is nondecreasing. \label{kle}
\end{lemma}

To prove the Liouville-type theorem in (\ref{pus}) for the critical
case
$$
\frac{2(p+1)}{pq-1}=\tilde{n}-1, \ \mbox{and} \ \
\frac{2(q+1)}{pq-1}=\tilde{n}-1,
$$
we will compare the supersolutions $u, v$ with an explicit
subsolution of the equation
$$ -\mathcal{M}^+_{\lambda, \Lambda}(D^2\phi)=(
\frac{x_n}{|x|^{\tilde{n}}})^{\frac{\tilde{n}+1}{\tilde{n}-1}}.$$
Such a subsolution is constructed as follows.
\begin{lemma}
There exist positive constants $e,f>0$ and $r_0\geq 1$, which only
depend on $\lambda, \Lambda$ and $n$ such that the function
$$\Gamma(x)=\frac{x_n}{|x|^{\tilde{n}}}(e ln|x|+f
(\frac{x_n}{|x|})^2)$$ satisfies
\begin{equation}
-\mathcal{M}^+_{\lambda,\Lambda}(D^2\Gamma)\leq
(\frac{x_n}{|x|^{\tilde{n}}})^{\frac{\tilde{n}+1}{\tilde{n}-1}}
\quad \mbox{in} \  \mathbb R^n_+\backslash \mathbb B_{r_0}
\label{pur}
\end{equation}
 in the classical sense.
 \label{lll}
\end{lemma}

\begin{proof}
We consider
$$\Gamma_1(x):=\frac{x_n}{|x|^{\tilde{n}}} ln|x|$$
and
$$ \Gamma_2(x):=\frac{x_n^3}{|x|^{\tilde{n}+2}}.$$
Then $\Gamma(x)=e \Gamma_1(x)+f\Gamma_2(x).$ From the property of
the Pucci maximal operator, it yields that
\begin{equation}
-\mathcal{M}^+_{\lambda,\Lambda}(D^2\Gamma)\leq -e
\mathcal{M}^+_{\lambda,\Lambda}(D^2\Gamma_1)-f\mathcal{M}^-_{\lambda,\Lambda}(D^2\Gamma_2).
\label{prop}
\end{equation}
In order to obtain (\ref{pur}), we estimate the terms on the right
hand side of (\ref{prop}), respectively. As far as $\Gamma_1$ is
concerned, direct calculations show that
\begin{equation}
\begin{array}{ll}
D^2\Gamma_1(x)=&\frac{x_n}{|x|^{\tilde{n}+2}}\{[(\tilde{n}+2)\tilde{n}ln|x|-2(\tilde{n}+1)]
\frac{x}{|x|}\otimes\frac{x}{|x|}+(1-\tilde{n}ln|x|)e_n\otimes e_n
\nonumber\\
\\ &
+(1-\tilde{n}ln|x|)\frac{|x|}{x}(\frac{x}{|x|}\otimes
e_n+e_n\otimes\frac{x}{|x|})-(\tilde{n}ln|x|-1)I_n \}.\nonumber\\
\end{array}
\end{equation}
Recall that $\tilde{n}=\frac{\lambda}{\Lambda}(n-1)+1$. According to
Lemma \ref{tec}, the eigenvalue $\mu_1, \mu_2, \cdots, \mu_n$ of
$D^2\Gamma_1$ are
$$
\mu_1=\frac{x_n}{|x|^{\tilde{n}+2}}\frac{\tilde{n}^2ln|x|-3\tilde{n}ln|x|-2\tilde{n}+3+\sqrt{D}}{2},
$$
$$
\mu_2=\frac{x_n}{|x|^{\tilde{n}+2}}\frac{\tilde{n}^2ln|x|-3\tilde{n}ln|x|-2\tilde{n}+3-\sqrt{D}}{2},
$$
$$\mu_i=-\frac{x_n}{|x|^{\tilde{n}+2}}(\tilde{n}ln|x|-1), \quad
3\leq i\leq n,$$ where
\begin{equation}
\begin{array}{lll}
D
&=&[\tilde{n}(\tilde{n}+2)ln|x|-2(\tilde{n}+1)+3(1-\tilde{n}ln|x|)]^2
\nonumber \\
\\
&&+4(1-\frac{x_n^2}{|x|^2})\{(1-\tilde{n}ln|x|)^2\frac{|x|^2}{x_n^2}-
[(\tilde{n}+2)\tilde{n}ln|x|-2(\tilde{n}+1)](1-\tilde{n}ln|x|)\}
\nonumber
\\
\\
&\geq & [(\tilde{n}+2)(\tilde{n}ln|x|-2)+3(1-\tilde{n}ln|x|)]^2 \nonumber \\
\\
 &&+4(1-\frac{x_n^2}{|x|^2})\{(1-\tilde{n}ln|x|)^2\frac{|x|^2}{x_n^2}-(\tilde{n}+2)(\tilde{n}ln|x|-2)(1-\tilde{n}ln|x|)\}
 \nonumber \\
 \\
&\geq &
[(\tilde{n}ln|x|-2)(\tilde{n}-1)]^2+4(1-\frac{x_n^2}{|x|^2})(\tilde{n}ln|x|-2)^2[\frac{|x|^2}{x_n^2}+(\tilde{n}+2)].\\
\end{array}
\end{equation}
Hence
$$\sqrt{D}\geq (\tilde{n}ln|x|-2)(\tilde{n}-1).$$
For $r>r_0$, it follows that $\mu_1\geq 0$ and $\mu_i\leq 0$ for
$2\leq i\leq n$, where $r_0$ depends on $\Lambda$, $\lambda$ and
$n$. Therefore, one has
\begin{equation}
\begin{array}{lll}
\mathcal{M}^+_{\lambda,\Lambda}(D^2\Gamma_1)&=&\Lambda\mu_1+\lambda\sum^{n}_{i=2}\mu_i
\nonumber \\
\\
&=&\frac{x_n}{|x|^{\tilde{n}+2}}\{\frac{(\Lambda+\lambda)(\tilde{n}^2ln|x|-3\tilde{n}ln|x|-2\tilde{n}+3)+
(\Lambda-\lambda)\sqrt{D}}
{2} \nonumber \\
\\
&&-(n-2)\lambda(\tilde{n}ln|x|-1)\} \nonumber \\
\\
&\geq&
\frac{x_n}{|x|^{\tilde{n}+2}}\frac{(\Lambda+\lambda)(-2\tilde{n}+3)-2(\Lambda-\lambda)(\tilde{n}-1)+2(n-2)\lambda}{2}
\\
\\
&=&-\frac{x_n}{|x|^{\tilde{n}+2}}\frac{2\lambda
n-\Lambda-\lambda}{2} \nonumber \\ \\ & =& -c_1
\frac{x_n}{|x|^{\tilde{n}+2}},
\end{array}
\end{equation}
where $c_1=\frac{2\lambda n-\Lambda-\lambda}{2}$. Since
$\tilde{n}=\frac{\lambda}{\Lambda}(n-1)+1\geq 2$, we get $c_1> 0$.
By the argument in Theorem 2.3 in \cite{L}, we have
\begin{equation}
\begin{array}{lll}
\mathcal{M}^-_{\lambda,\Lambda}(D^2\Gamma_2)&\geq& \frac{\lambda
x_n^3}{|x|^{\tilde{n}+4}}\{(\tilde{n}+2)[\tilde{n}-3-\frac{\Lambda}{\lambda}(n-1)]
+3(3-\frac{\Lambda}{\lambda})\frac{|x|^2}{x_n^2}\} \nonumber \\
\\&\geq&\frac{\lambda
x_n^3}{|x|^{\tilde{n}+4}}\{\tilde{n}[\tilde{n}-3-\frac{\Lambda}{\lambda}(n-1)]+2[\tilde{n}-\frac{\Lambda}{\lambda}(n-1)]
+3(1-\frac{\Lambda}{\lambda})\frac{|x|^2}{x_n^2}\}
\nonumber \\
\\&= & \frac{\lambda
x_n^3}{|x|^{\tilde{n}+4}}\{\tilde{n}(\frac{\lambda}{\Lambda}-\frac{\Lambda}{\lambda})(n-1)-2\frac{\Lambda}{\lambda}(n-1)
+3(1-\frac{\Lambda}{\lambda})\frac{|x|^2}{x_n^2}\}
\nonumber \\
\\&\geq & -\frac{
x_n^3}{|x|^{\tilde{n}+4}}\{c_2-c_3\frac{|x|^2}{x_n^2}\},
\end{array}
\end{equation}
where
$c_2=\tilde{n}(\frac{\Lambda^2-\lambda^2}{\Lambda})(n-1)+2\Lambda(n-1)$
and $c_3=3(\Lambda-\lambda)$. Then setting $f=c_2^{-1}$ and
$e=\frac{c_3}{c_2 c_1}$, we obtain

\begin{equation}
\begin{array}{lll}
-\mathcal{M}^+_{\lambda,\Lambda}(D^2\Gamma)&\leq&
-e\mathcal{M}^+_{\lambda,\Lambda}(D^2\Gamma_1)-f\mathcal{M}^-_{\lambda,\Lambda}(D^2\Gamma_2)
\nonumber \\
\\&\leq& ec_1\frac{x_n}{|x|^{\tilde{n}+2}}+fc_2\frac{
x_n^3}{|x|^{\tilde{n}+4}}-fc_3\frac{ x_n}{|x|^{\tilde{n}+2}}
\nonumber \\
\\&\leq& \frac{ x_n^3}{|x|^{\tilde{n}+4}},
\end{array}
\end{equation}
Furthermore, since $\tilde{n}\geq 2$, a direct calculation yields that
\begin{equation}
\begin{array}{lll}
 -\mathcal{M}^+_{\lambda,\Lambda}(D^2\Gamma)&\leq&  \frac{
x_n^3}{|x|^{\tilde{n}+4}} \nonumber \\
\\&\leq&(\frac{x_n}{|x|^{\tilde{n}}})^{\frac{\tilde{n}+1}{\tilde{n}-1}}.
\end{array}
\end{equation}
Hence the lemma is completed.
\end{proof}

Now we present the proof of Theorem \ref{th4}. Our idea is inspired
by the work in \cite{L}.
\begin{proof}[Proof of Theorem \ref{th4}]
By the strong maximal principle (i.e. Lemma \ref{hopf}), we may
assume that $u, v>0$ in $\mathbb R^n_+$. Let us rescale the
supersolutions in ($\ref{pus}$). For every $r>0$, we set
$$u_r(x)=u(rx),$$
$$v_r(x)=v(rx).$$
Then $u_r, v_r>0$ are supersolutions for
\begin{equation}
\left \{ \begin{array}{ll}
\mathcal{M}^+_{\lambda, \Lambda}(D^2u_r)+r^2v_r^p=0   \quad \quad  &\mbox{in}\ \mathbb R^n_+, \\
\mathcal{M}^+_{\lambda, \Lambda}(D^2v_r)+r^2u_r^q=0   \quad \quad
&\mbox{in}\ \mathbb
R^n_+. \\
\end{array}
\right. \label{pus1}
\end{equation}
Next we will choose appropriate test functions for supersolutions
$u_r, v_r$. Selecting a smooth, concave, nonincreasing function:
$\eta: [0, \ +\infty)\to R$ satisfying
\begin{equation}
\eta(t)=\left \{\begin{array}{ll} 1 \quad \quad &\mbox{for} \ 0\leq
t\leq 1/2, \\
>0 & \mbox{for} \ \   1/2<t<3/4, \\
\leq 0 &\mbox{for} \ t\geq 3/4.
\end{array}
\right.
\end{equation}
Fixed a point $a=(0,1)$. Here $\mathbb B_r(a)$ is a ball centered at
$a$ with radius $r$. Let
$$U(x)=(\inf_{\mathbb B_{1/2}(a)} u_r)\eta(|x-a|),$$
$$V(x)=(\inf_{\mathbb B_{1/2}(a)} v_r)\eta(|x-a|).$$
It is easy to see that $u_r\geq U$ in $\overline{\mathbb
B_{1/2}}(a)$, $u_r=U$ at some point on $\partial \mathbb B_{1/2}(a)$
by the maximum principle (i.e. Lemma \ref{com}) and $u_r> U$ outside
$\mathbb B_{3/4}(a)$. By the same observation, $v_r\geq V$ in
$\overline{\mathbb B_{1/2}}(a)$, $v_r=V$ at some point on $\partial
\mathbb B_{1/2}(a)$ and $v_r> V$ outside $\mathbb B_{3/4}(a)$.
Therefore, the infimum of $u_r-U, v_r-V$ is non-positive and
achieved at $x_1, x_2$ in $\mathbb B_{3/4}(a)\backslash\mathbb
B_{1/2}(a)$, respectively. From the definition of a viscosity solution
and taking into account that $U, V$ are test functions for $u_r,
v_r$, respectively, it yields that
\begin{equation}v_r^p(x_1)\leq \frac{C_1}{r^2} \inf_{\mathbb
B_{1/2}(a)} u_r \label{es1}
\end{equation}
and
\begin{equation}
u_r^q(x_2)\leq \frac{C_1}{r^2} \inf_{\mathbb B_{1/2}(a)} v_r,
\label{es2}
\end{equation} where
$$C_1=\sup_{\mathbb B_{3/4}(a)}(-\mathcal{M}^+_{\lambda, \Lambda}(D^2\eta))=\sup_{\mathbb B_{3/4}(a)}
(-\lambda\triangle\eta)=-\lambda\inf_{t\in [1/2, \
3/4]}(\eta''(t)+(n-1) t^{-1}\eta').
$$
Since $u_r(x)$ and $ v_r(x)$ are also supersolutions for
$\mathcal{M}^+_{\lambda, \Lambda}(D^2 u_r)=0$ and
$\mathcal{M}^+_{\lambda, \Lambda}(D^2 v_r)=0$, respectively, the
monotonicity property ( see \cite{CL} ) implies that
\begin{equation}
\inf_{\mathbb B_{1/2}(a)}u_r\leq C \inf_{\mathbb B_{3/4}(a)}u_r,
\label{es3}
\end{equation}
\begin{equation}
\inf_{\mathbb B_{1/2}(a)}v_r\leq C \inf_{\mathbb B_{3/4}(a)}v_r.
\label{es4}
\end{equation}
Furthermore, From (\ref{es1})-(\ref{es4}), we get
$$
(\inf_{\mathbb B_{3/4}(a)}v_r)^p\leq v_r^p(x_1)\leq \frac{C_1}{r^2}
\inf_{\mathbb B_{1/2}(a)} u_r\leq
 \frac{C}{r^2} \inf_{\mathbb B_{3/4}(a)}u_r\leq  \frac{C}{r^2} (\frac{C_1}{r^2}\inf_{\mathbb
 B_{1/2}(a)}v_r)^{\frac{1}{q}}\leq \frac{C}{r^{2(1+\frac{1}{q})}}(\inf_{\mathbb
 B_{3/4}(a)}v_r)^{\frac{1}{q}},$$
that is,
\begin{equation}
(\inf_{\mathbb B_{3/4}(a)}v_r)\leq \frac{C}{r^\frac{2(q+1)}{pq-1}}.
\label{cr1}
\end{equation}
Similar argument indicates that
$$
(\inf_{\mathbb B_{3/4}(a)}u_r)^q\leq u_r^q(x_1)\leq \frac{C_1}{r^2}
\inf_{\mathbb B_{1/2}(a)} v_r\leq
 \frac{C}{r^2} \inf_{\mathbb B_{3/4}(a)}v_r\leq \frac{C}{r^2} (\frac{C_1}{r^2}\inf_{\mathbb
 B_{1/2}(a)}u_r)^{\frac{1}{p}}\leq \frac{C}{r^{2(1+\frac{1}{p})}}(\inf_{\mathbb
 B_{3/4}(a)}u_r)^{\frac{1}{p}},$$
that is,
\begin{equation}
(\inf_{\mathbb B_{3/4}(a)}u_r)\leq \frac{C}{r^\frac{2(p+1)}{pq-1}}.
\label{cr2}
\end{equation}
If $pq=1$, A contradiction is obviously arrived. We readily infer
that $u\equiv v\equiv 0$.

While $pq>1$, we observe that
\begin{equation}
\inf_{\mathbb B_{3/4}(a)}v_r=\inf_{\mathbb B_{3r/4}(ar)}v\geq
\frac{r}{4}\inf_{\mathbb
B_{3r/4}(ar)}\frac{v}{x_n}\geq\frac{r}{4}\inf_{\mathbb
B_{2r}}\frac{v}{x_n}=\frac{r}{4}m_v(2r), \label{pq2}
\end{equation}
\begin{equation}
\inf_{\mathbb B_{3/4}(a)}u_r=\inf_{\mathbb B_{3r/4}(ar)}u\geq
\frac{r}{4}\inf_{\mathbb
B_{3r/4}(ar)}\frac{u}{x_n}\geq\frac{r}{4}\inf_{\mathbb
B_{2r}}\frac{u}{x_n}=\frac{r}{4}m_u(2r). \label{pq1}
\end{equation}
From (\ref{cr1}) and (\ref{pq2}), we obtain
\begin{equation}r^{\tilde{n}}m_v(r)\leq
\frac{C}{r^{\frac{2(q+1)}{pq-1}+1-\tilde{n}}}. \label{tt3}
\end{equation}
 By (\ref{cr2}) and (\ref{pq1}), we have
\begin{equation}
r^{\tilde{n}}m_u(r)\leq
\frac{C}{r^{\frac{2(p+1)}{pq-1}+1-\tilde{n}}}. \label{tt1}
\end{equation} If
$$\frac{2(p+1)}{pq-1}>\tilde{n}-1 \ \ \mbox{or} \ \
\frac{2(q+1)}{pq-1}>\tilde{n}-1,$$ then $r^{\tilde{n}}m_u(r)\to 0$
or $r^{\tilde{n}}m_v(r)\to 0$ as $r\to\infty$. Hence Lemma \ref{kle}
shows that $u\equiv 0$ or $v\equiv 0$. From the structure of fully nonlinear
elliptic equation systems, we obtain $u\equiv 0$  and  $v\equiv 0$
in either of the cases.

Next we study the critical case that
$$\frac{2(p+1)}{pq-1}=\tilde{n}-1 \ \ \mbox{and} \ \
\frac{2(q+1)}{pq-1}=\tilde{n}-1 .$$ It is easy to check that
$p=q=\frac{\tilde{n}+1}{\tilde{n}-1}$. In this case,  (\ref{tt3})
and (\ref{tt1}) become
\begin{equation}
r^{\tilde{n}}m_v(r)\leq C \quad \quad \forall r>0 \label{fuk}
\end{equation}
and
\begin{equation}
r^{\tilde{n}}m_u(r)\leq C \quad \quad \forall r>0. \label{tur}
\end{equation}
Thanks to the monotonicity property of $r^{\tilde{n}}m_u(r)$ in
Lemma \ref{kle},
$$ r^{\tilde{n}}m_u(r)\geq r_0^{\tilde{n}}m_u(r_0) \quad \
\mbox{for} \ r\geq r_0.$$ Then
\begin{equation}
u(x)\geq C \frac{x_n}{r^{\tilde{n}}} \quad \ \mbox{for} \  x\in
\mathbb R^n_+\backslash \mathbb B_{r_0}. \label{tll}
\end{equation}

With the aid of (\ref{tll}),
\begin{equation}
-\mav\geq
C(\frac{x_n}{r^{\tilde{n}}})^\frac{\tilde{n}+1}{\tilde{n}-1}, \quad
\ \forall x\in\mathbb R^n_+\backslash \mathbb B_{r_0}.
\end{equation}
Taking into account of  Lemma \ref{lll},
\begin{equation}
-\mathcal{M}^+_{\lambda, \Lambda}(D^2(\gamma \Gamma))\leq -\mav
\end{equation}
is satisfied by appropriately chosen  $\gamma$. Choosing $$\gamma\leq
m_v(r_0)\frac{r_0^{\tilde{n}_1}}{e ln r_0+f},$$ we have
$$\gamma\Gamma(x)\leq v(x) \quad \ \mbox{on} \ \partial\mathbb
B_{r_0}.$$

Fixed any $\epsilon>0$, let $R>0$ be so large that
$$\gamma\Gamma(x)\leq \epsilon \quad \ \mbox{for} \ \mathbb
R^n_+\backslash \mathbb B_R.$$

The comparison principle in Lemma \ref{com} for $\gamma\Gamma(x)$
and $v(x)+\epsilon$ in $\mathbb B_R\backslash\mathbb B_{r_0}$ shows
that
$$\gamma\Gamma(x)\leq v(x)+\epsilon.$$ In addition, let $R\to \infty$
and then $\epsilon\to 0$, we have
 $$\gamma\Gamma(x)\leq v(x) \quad \ \forall x\in
 \mathbb R^n_+\backslash \mathbb B_{r_0}.$$
 From the explicit form of $\Gamma(x)$,
$$v(x)\geq C\frac{x_n}{|x|^{\tilde{n}}}ln|x| \quad \ \forall x\in
 \mathbb R^n_+\backslash \mathbb B_{r_0},$$
 which implies that
 $$m_v(r) r^{\tilde{n}}\geq Cln r   \quad \ \forall r\geq r_0.$$ It contradicts the bound in (\ref{fuk}).
 The theorem is thus accomplished.

\end{proof}

The proof of Corollary \ref{cor2} is the consequence of the above
arguments and estimates in \cite{L}. We omit it here.

\end{document}